\documentclass[11pt,twoside]{amsart}

\usepackage{longtable}
\usepackage[all]{xy}
\usepackage{amsmath}
\usepackage{amsfonts}
\usepackage{amssymb}
\usepackage{amsthm}
\usepackage{color}
\usepackage{enumerate}
\usepackage{hyperref}
\usepackage{fullpage}
\usepackage{graphicx}
\usepackage{url}

\usepackage{tikz}
\usetikzlibrary{arrows,shapes,positioning,calc,patterns,cd}
\tikzcdset{row sep/normal=1.50em}
\tikzcdset{column sep/normal=2.25em}
\tikzcdset{cells={outer sep=-2pt}}

\theoremstyle{definition}
\newtheorem{defn}{Definition}[section]

\theoremstyle{plain}
\newtheorem{thm}[defn]{Theorem}

\newtheorem{lem}[defn]{Lemma}
\newtheorem{prop}[defn]{Proposition}
\newtheorem{cor}[defn]{Corollary}

\def\C{\ensuremath{\mathbb{C}}}
\def\D{\ensuremath{\mathbb{D}}}
\def\F{\ensuremath{\mathbb{F}}}

\def\P{\ensuremath{\mathbb{P}}}

\def\R{\ensuremath{\mathbb{R}}}
\def\Z{\ensuremath{\mathbb{Z}}}

\def\FF{\ensuremath{\mathcal F}}

\def\HH{\ensuremath{\mathcal H}}
\def\II{\ensuremath{\mathcal I}}

\def\OO{\ensuremath{\mathcal O}}

\def\TT{\ensuremath{\mathcal T}}

\def\ch{\mathop{\mathrm{ch}}\nolimits}

\def\Coh{\mathop{\mathrm{Coh}}\nolimits}

\def\Db{\mathop{\mathrm{D}^{\mathrm{b}}}\nolimits}
\def\deg{\mathop{\mathrm{deg}}\nolimits}

\def\dim{\mathop{\mathrm{dim}}\nolimits}

\def\Ext{\mathop{\mathrm{Ext}}\nolimits}

\def\hom{\mathop{\mathrm{hom}}\nolimits}
\def\Hom{\mathop{\mathrm{Hom}}\nolimits}

\def\RlHom{\mathop{\mathbf{R}\mathcal Hom}\nolimits}

\def\Pic{\mathop{\mathrm{Pic}}\nolimits}

\def\td{\mathop{\mathrm{td}}\nolimits}

\def\into{\ensuremath{\hookrightarrow}}
\def\onto{\ensuremath{\twoheadrightarrow}}

\begin{document}

\title{Discriminants of stable rank two sheaves on some general type surfaces}

\author{Benjamin Schmidt}
\address{The University of Texas at Austin, Department of Mathematics, 2515 Speedway, RLM 8.100, Austin, TX 78712, USA}
\email{schmidt@math.utexas.edu}
\urladdr{https://sites.google.com/site/benjaminschmidtmath/}

\author{Benjamin Sung}
\address{Northeastern University, Department of Mathematics, 360 Huntington Avenue, Boston, MA 02115-5000, USA}
\email{b.sung@northeastern.edu}

\keywords{Stable sheaves, Stability conditions, Derived categories}

\subjclass[2010]{14J60 (Primary); 14D20, 14F05 (Secondary)}

\begin{abstract}
We prove sharp bounds on the discriminants of stable rank two sheaves on surfaces in three-dimensional projective space. The key technical ingredient is to study them as torsion sheaves in projective space via tilt stability in the derived category. We then proceed to describe the surface itself as a moduli space of rank two vector bundles on it. Lastly, we give a proof of the Bogomolov inequality for semistable rank two sheaves on integral surfaces in three-dimensional projective space in all characteristics.
\end{abstract}

\maketitle


\section{Introduction}

A fundamental problem in the theory of stable sheaves on surfaces is to understand their possible Chern characters. A first major step was the \emph{Bogomolov inequality} (see \cite{Bog78:inequality}). If the ground field has characteristic zero, it says that any semistable sheaf $E$ has positive discriminant $\Delta(E) \geq 0$ (see Theorem \ref{thm:bogomolov_inequality} for a definition of $\Delta(E)$). 

A complete classification of Chern characters of stable sheaves has been obtained in various special cases. However, when it comes to general type surfaces, almost nothing is known beyond the Bogomolov inequality. We prove the following statement about stable rank two sheaves on general type surfaces in $\P^3$.

\newtheorem*{thm:rank_two_bounds_surface}{Theorem \ref{thm:rank_two_bounds_surface}}
\begin{thm:rank_two_bounds_surface}
Let $S \subset \P^3$ be a very general smooth projective surface of degree $d \geq 5$ over an algebraically closed field $\F$, and let $H$ be the hyperplane section on $S$. Further assume $E \in \Coh(S)$ is a slope-stable sheaf with $\ch_0(E) = 2$.
\begin{enumerate}
\item If $\ch_1(E) = -H$, then
\[
\Delta(E) \geq 3d^2 - 4d.
\]
Equality implies that $h^0(E(H)) \geq 3$. Moreover, equality can be obtained for non-trivial extensions
\[
0 \to \OO_S(-H) \to E \to \II_Z \to 0,
\]
where $Z$ is a zero-dimensional subscheme of length $d - 1$ contained in a line.
\item If $\ch_1(E) = 0$, then 
\[
\Delta(E) \geq 4d^2.
\]
Equality can be obtained for non-trivial extensions
\[
0 \to \OO_S(-H) \to E \to \II_Z(H) \to 0,
\]
where $Z$ is a zero-dimensional subscheme of length $2d$ contained in two non-intersecting lines such that $d$ points are contained in each line.
\end{enumerate}
\end{thm:rank_two_bounds_surface}

The case $d = 5$ and $\ch_1(E) = -H$ was previously shown by Mestrano and Simpson (see \cite{MS11:rank_two_quinticI, MS18:rank_two_quinticII}). We use the fact that $S$ is very general only in two instances. Firstly, it is necessary to have $\Pic(X) = \Z \cdot H$ for the construction of semistable objects in which the discriminant reaches the bound. Secondly, when proving the bound, we use it in Lemma \ref{lem:torsion_destabilized_by_rank_one} to avoid the case in which $X$ contains a line. It would be interesting to see whether the existence of such a line can be used to construct a counterexample to the conclusion of the theorem.

\newtheorem*{cor:moduli_on_surface}{Corollary \ref{cor:moduli_on_surface}}
\begin{cor:moduli_on_surface}
Let $S \subset \P^3$ be a very general surface of degree $d \geq 5$ over an algebraically closed field of characteristic zero, and let $H$ be the hyperplane section on $S$. The moduli space of semistable rank two sheaves on $S$ with $\ch_1(E) = -H$ and $\Delta(E) = 3d^2 - 4d$ is given by $S$.
\end{cor:moduli_on_surface}

The Bogomolov inequality does not hold for arbitrary surfaces in positive characteristic. The major positive results in this case are due to \cite{Lan04:positive_char}. We prove that the inequality holds for semistable rank two sheaves on surfaces in $\P^3$ regardless of the characteristic of the field and for arbitrary integral surfaces. 

\newtheorem*{cor:bogomolov}{Corollary \ref{cor:bogomolov}}
\begin{cor:bogomolov}
Let $S \subset \P^3$ be an integral hypersurface, and let $H$ be the hyperplane section. If $E$ is a rank two slope-semistable torsion free sheaf on $S$, then $\Delta(E) \geq 0$.
\end{cor:bogomolov}

A quick remark is in order. If $X$ is singular, it is not entirely clear how to define $\Delta(E)$. For the purposes of this statement, we simply define the Chern characters of $E$ on $S$ via the Chern characters of $E$ in $\P^3$ by formally applying the Grothendieck-Riemann-Roch formula.

\subsection{Ingredients of the proof}

The proof of these statements is based on stability in the derived category of $\P^3$. More precisely, we are using the notion of \emph{tilt stability}. It can be thought of as a generalization of the classical notion of slope stability for sheaves on surfaces. It roughly amounts to replacing the category of coherent sheaves with a different abelian category $\Coh^{\beta}(\P^3)$ embedded in the bounded derived category $\Db(\P^3)$ and the classical slope with a new slope function $\nu_{\alpha,\beta}$. Everything depends on two real parameters $\alpha, \beta \in \R$, $\alpha > 0$. 

Let $E$ be a rank two stable sheaf on a surface $i: S \into \P^3$. The starting point is that for $\alpha \gg 0$, the sheaf $i_*E$ is $\nu_{\alpha,\beta}$-stable. If a wall for $i_* E$ is induced by a short exact sequence $0 \to F \to i_* E \to G \to 0$, then bounds on $\ch_3(F)$ and $\ch_3(G)$ lead to a bound on $\ch_{2, S}(E)$.

There is one issue with this approach. There are sheaves in $\P^3$ with the same first and second Chern character as $i_* E$ that are supported on a surface of degree $2d$. The third Chern characters of these objects satistfy much weaker bounds. The point of Proposition \ref{prop:rank_two_destabilizes_rank_two} and Corollary \ref{cor:rank_two_destabilizes_rank_two} is to impose restrictions on walls that take this issue into account.

\subsection{Further motivation}
One other motivation for understanding sharper bounds on discriminants is in the construction of Bridgeland stability conditions for higher dimensional varieties. The existence of such stability conditions is not known in general, and in fact they have only recently been constructed for quintic threefolds in \cite{Li18:bg3_quintic}. The main technique employed was to reduce a stronger Bogomolov-Gieseker type inequality for the second Chern character to an inequality on a complete intersection of a quintic and a quadric $S_{2,5} = X_{5} \cap Q_{1} \subset \mathbb{P}^{4}$ via restriction arguments. This is subsequently established by a stronger Clifford type bound on an embedded curve $C_{2,2,5} = S_{2,5} \cap Q_{2} \subset \mathbb{P}^{4}$ via inequalities on pushforwards to torsion sheaves on $Q_{1} \cap Q_{2} \subset \mathbb{P}^{4}$. This result is then used to prove the conjectured generalized Bogomolov inequality from \cite{BMT14:stability_threefolds} in many cases. We remark that our approach parallels the aforementioned technique in a different situation with a hyperplane section of the quintic threefold and using inequalities on $\mathbb{P}^{3}$. 


\subsection{Structure of the article}

In Section \ref{sec:prelim} we recall basic notions of stability. The proof of the main theorem requires showing Chern character bounds for some other semistable objects in $\P^3$. These bounds are obtained in Section \ref{sec:some_bounds}. In Section \ref{sec:walls} we prove results that reduce the number of possible walls for rank two sheaves supported on surfaces in $\P^3$. These statements also lead to a proof of Corollary \ref{cor:bogomolov}. Finally, in Section \ref{sec:main} we prove Theorem \ref{thm:rank_two_bounds_surface} and Corollary \ref{cor:moduli_on_surface}.

\subsection*{Acknowledgments}

We would like to thank Tom Bridgeland, Izzet Coskun, Sean Keel, and Emanuele Macr\`i for very useful discussions. We also thank the referee for useful comments. B. Schmidt is supported by an AMS-Simons Travel Grant. B. Sung has been supported by NSF RTG Grant DMS-1645877 and the NSF Graduate Research Fellowship under grant DGE-1451070.

\subsection*{Notation}

\begin{center}
   \begin{tabular}{ r l }
     $\F$ & an algebraically closed field \\
     $X$ & smooth projective variety over $\F$ \\
     $n$ & $\dim X$ \\
     $H$ & fixed ample divisor on $X$ \\
     $\Db(X)$ & bounded derived category of coherent sheaves on $X$ \\
     $\HH^{i}(E)$ & the $i$-th cohomology group of a complex $E \in \Db(X)$ \\
     $H^i(E)$ & the $i$-th sheaf cohomology group of a complex $E \in \Db(X)$ \\
     $\D(\cdot)$ & the shifted derived dual $\RlHom(\cdot, \OO_X)[1]$ \\
     $\ch_X(E) = \ch(E)$ & Chern character of an object $E \in \Db(X)$  \\
     $\ch_{\leq l, X}(E) = \ch_{\leq l}(E)$ & $(\ch_0(E), \ldots, \ch_l(E))$ \\
     $H \cdot \ch_X(E) = H \cdot \ch(E)$ & $\left(H^n \cdot \ch_0(E), H^{n-1} \cdot \ch_1(E), \ldots, \ch_n(E)\right)$ \\
     $\td_{\P^3} = \td(T_{\P^3})$ & $\left(1,2,\frac{11}{6},1\right)$ \\
     $\td_S = \td(T_S)$ & $\left(1, \left(2 - \frac{d}{2}\right)H, \frac{d^3}{6} - d^2 + \frac{11d}{6} \right)$ for a surface $S \subset \P^3$ of degree $d$.
   \end{tabular}
\end{center}

\section{Background in stability}
\label{sec:prelim}

We will explain various notions of stability in this section. Let $X$ be a smooth projective variety over an algebraically closed field $\F$, $H$ be the class of an ample divisor on $X$, and $n = \dim(X)$.

\subsection{Classical Notions}

If $E \in \Coh(X)$ is an arbitrary coherent sheaf, then its \emph{slope} is
\[
\mu(E) := \frac{H^{n-1} \cdot \ch_1(E)}{H^n \cdot \ch_0(E)}.
\]
If $\ch_0(E) = 0$, then we define $\mu(E) := +\infty$. We say that $E$ is \emph{slope-(semi)stable} if any non-zero subsheaf $A \subset E$ satisfies $\mu(A) < (\leq) \mu(E/A)$. This notion has good computational properties, but often we require a more flexible notion of stability. The idea is to introduce further tiebreaker functions in the case of equal slope. This is most easily stated in terms of polynomials. For $p, q \in \R[m]$ we define an order as follows.
\begin{enumerate}
\item If $\deg(p) < \deg(q)$, then $p > q$.
\item Let $d = \deg(p) = \deg(q)$, and let $a$, $b$ be the leading coefficients in $p$, $q$ respectively. Then $p < (\leq) q$ if $\tfrac{p(m)}{a} < (\leq) \tfrac{q(m)}{b}$ for all $m \gg 0$. 
\end{enumerate}

\begin{defn}
\begin{enumerate}
\item Let $E$ be a coherent sheaf and let $k$ be an integer with $1 \leq k \leq n$. Let $P_n(E, m) := \chi(E(m))$ be the Hilbert polynomial of $E$, and define $\alpha_i(E)$ via
\[
P_n(E, m) = \sum_{i = 0}^n \alpha_i(E) m^i.
\]
We define
\[
P_k(E, m) := \sum_{i = n-k}^n \alpha_i(E) m^i.
\]
\item A coherent sheaf $E$ is called \emph{$k$-GS-(semi)stable} if for any non-zero subsheaf $A \into E$ the inequality $P_k(A, m) < (\leq) P_k(E/A, m)$ holds.
\item A coherent sheaf $E$ is called \emph{GS-(semi)stable} if it is $n$-GS-(semi)stable.
\end{enumerate}
\end{defn}

A few remarks are in order. A Riemann-Roch calculation shows that the notion of $1$-GS-stability is the same as slope stability. Gieseker first introduced GS-stability for torsion-free sheaves. Later, Simpson generalized it to torsion sheaves. The notion of $k$-GS-stability is lesser-known, but turns out to be relevant for $k = 2$ in the context of Bridgeland stability on threefolds. We will discuss this further below. A very simple argument yields the following relations between these notions.

\centerline{
\xymatrix{
\text{slope-stable} \ar@{=>}[r] & \text{$2$-GS-stable} \ar@{=>}[r] & \ldots \ar@{=>}[r] & \text{$(n-1)$-GS-stable}  \ar@{=>}[r] & \text{GS-stable} \ar@{=>}[d] \\
\text{slope-semistable} & \text{$2$-GS-semistable} \ar@{=>}[l] & \ldots \ar@{=>}[l] & \text{$(n-1)$-GS-semistable} \ar@{=>}[l] &\text{GS-semistable} \ar@{=>}[l]
}}

The Chern characters of semistable objects satisfy non-trivial inequalities. The most famous one is the Bogomolov inequality (\cite{Bog78:inequality}).

\begin{thm}[Bogomolov inequality]
\label{thm:bogomolov_inequality}
Assume that $\F$ has characteristic $0$. If $E$ is a slope-semistable sheaf $E \in \Coh(X)$, then
\[
\Delta(E) := (H^{n-1} \cdot \ch_1(E))^2 - 2 (H^n \cdot \ch_0(E)) (H^{n-2} \cdot \ch_2(E)) \geq 0.
\]
\end{thm}

In positive characteristic, this theorem is only true for special varieties. It holds for example on $\P^3$ and abelian threefolds (see \cite{Lan04:positive_char} for more details). We will only use it for $\P^3$. Therefore, most results in this article also hold in positive characteristic.

\subsection{Tilt stability}

In order to prove our main Theorem, we would like to obtain inequalities on third Chern characters. It turns out that the derived category serves as a natural setting for this. Classically, algebraic geometers have varied the slope function $\mu$ via the polarization $H$. Bridgeland's brilliant idea was to view the abelian category $\Coh(X)$ as a variable input, and to change it to different hearts of bounded t-structures. He introduced \emph{tilt stability} for K3 surfaces in \cite{Bri08:stability_k3}. After work by Arcara-Bertram \cite{AB13:k_trivial} on arbitrary surfaces, Bayer-Macr\`i-Toda (\cite{BMT14:stability_threefolds}) defined the notion for threefolds which we will proceed to explain.

If $\beta$ is any real number, then $\ch^{\beta}$ is defined to be $e^{-\beta H} \cdot \ch$. If $\beta \in \Z$ and $E \in \Db(X)$, this is simply $\ch(E(-\beta H))$. It expands as follows:

\begin{align*}
\ch^{\beta}_0 &=  \ch_0, \ \ch^{\beta}_1 = \ch_1 - \beta H \cdot \ch_0, \ \ch^{\beta}_2 = \ch_2 - \beta H \cdot \ch_1 + \frac{\beta^2}{2} H^2 \cdot \ch_0,\\
\ch^{\beta}_3 &= \ch_3 - \beta H \cdot \ch_2 + \frac{\beta^2}{2} H^2 \cdot \ch_1 - \frac{\beta^3}{6} H^3 \cdot \ch_0.
\end{align*}

In order to construct a new heart of a bounded t-structure, we need to tilt the category $\Coh(X)$. We refer to \cite{HRS96:tilting} for more details on tilting.

Let $\TT_{\beta} \subset \Coh(X)$ be the extension closure of all slope-semistable sheaves with slope strictly larger than $\beta$. Note that $\TT_{\beta}$ contains all torsion sheaves. By $\FF_{\beta}$, we denote the extension closure of all slope-semistable sheaves with slope smaller than or equal to $\beta$. The new heart is then the extension closure $\Coh^{\beta}(X) := \langle \FF_{\beta}[1],\TT_{\beta} \rangle$.

The next step is to define a slope function. We fix another real number $\alpha > 0$ and define the \emph{tilt-slope} as
\[
\nu_{\alpha, \beta} := \frac{H^{n-2} \cdot \ch^{\beta}_2 - \frac{\alpha^2}{2} H^n \cdot \ch^{\beta}_0}{H^{n-1} \cdot \ch^{\beta}_1}.
\]
As in the case of slope stability, an object $E \in \Coh^{\beta}(X)$ is \emph{tilt-(semi)stable} (or \emph{$\nu_{\alpha,\beta}$-(semi)stable}) if for any non-zero subobject $A \subset E$ the inequality $\nu_{\alpha, \beta}(A) < (\leq) \nu_{\alpha, \beta}(E/A)$ holds. In particular, the classical Bogomolov inequality also holds in this general setting.

\begin{thm}[{\cite[Corollary 7.3.2]{BMT14:stability_threefolds}}]
If $E \in \Coh^{\beta}(X)$ is $\nu_{\alpha, \beta}$-semistable, then $\Delta(E) \geq 0$.
\end{thm}

Proposition 14.2 in \cite{Bri08:stability_k3} is only stated for K3 surfaces, but the proof is no different in this more general setup. Together with \cite[Lemma 2.7]{BMS16:abelian_threefolds} we get the following proposition.

\begin{prop}
\label{prop:large_volume_limit}
Let $E \in \Coh^{\beta}(X)$. Then $E$ is $\nu_{\alpha, \beta}$-(semi)stable for $\alpha \gg 0$ and $\beta < \mu(E)$ if and only if it is $2$-GS-(semi)stable. If $E$ is $\nu_{\alpha, \beta}$-semistable for $\alpha \gg 0$ and $\beta > \mu(E)$, then $\HH^{-1}(E)$ is slope-semistable and $\HH^0(E)$ is zero or supported in dimension smaller than or equal to one.
\end{prop}

\begin{lem}[{\cite[Proposition 3.12]{LM16:examples_tilt}}]
\label{lem:h_minus_one_reflexive}
If $E \in \Coh^{\beta}(X)$ is tilt-semistable for $\beta > \mu(E)$, then $\HH^{-1}(E)$ is reflexive.
\end{lem}

\subsection{Wall and chamber structure}

We fix the lattice $\Lambda = \Z \oplus \Z \oplus \tfrac{1}{2} \Z$. Then we have a homomorphism $H \cdot \ch_{\leq 2}: K_0(X) \to \Lambda$ given by $x \mapsto (H^2 \cdot \ch_0(x), H \cdot \ch_1(x), \ch_2(x))$. The goal is to understand how the set of semistable objects varies with $(\alpha, \beta)$. If $v, w \in \Lambda$ are linearly independent, then the \textit{numerical wall} $W(v, w)$ is the set of $(\alpha, \beta)$ in the upper half plane such that $\nu_{\alpha, \beta}(v) = \nu_{\alpha, \beta}(w)$. Such a numerical wall $W$ is an \textit{actual wall for $v$} if the set of $\nu_{\alpha, \beta}$-semistable objects with class $v$ is different on both sides of $W$.

\begin{thm}[Wall Structure]
\label{thm:Bertram}
Let $v \in \Lambda$ be a class with $\Delta(v) \geq 0$.
\begin{enumerate}
  \item If $v_0 = 0$, then all numerical walls with respect to $v$ are semicircles with center at $\alpha = 0$ and $\beta = v_2/v_1$.
  \item If $v_0 \neq 0$, then there is a numerical wall given by $\beta = v_1/v_0$. All other walls are semicircles with center on the $\beta$-axis, whose apex lies on the hyperbola $\nu_{\alpha, \beta}(v) = 0$.
  \item Any two numerical walls with respect to $v$ have empty intersection.
  \item Assume there is an exact sequence of tilt-semistable objects $0 \to A \to E \to B \to 0$ with $\ch_{\leq 2}(E) = v$. If this sequence destabilizes $E$ along $W(\ch_{\leq 2}(A), \ch_{\leq 2}(B))$, then
  \[
  \Delta(A) + \Delta(B) \leq \Delta(E).
  \]
  Additionally, equality implies $H \cdot \ch_{\leq 2}(A) = 0$ or $H \cdot \ch_{\leq 2}(B) = 0$.
  \item If $\Delta(v) = 0$, then the only possible actual wall is $\beta = v_1/v_0$. In particular, if $v$ is the class of a line bundle, then there is no wall in tilt stability with respect to $v$.
\end{enumerate}
\end{thm}

The main result of \cite{Mac14:nested_wall_theorem} is essentially parts (i) - (iii). Appendix A in \cite{BMS16:abelian_threefolds} gives a proof of (iv) and (v). We write $\rho_W = \rho(v,w)$ for the \emph{radius} of $W(v, w)$ and denote the $\beta$-coordinate of its \emph{center} by $s_W = s(v,w)$.

The following lemma roughly says that walls induced by large rank subobjects are small.

\begin{lem}[{\cite[Proposition 8.3]{CH16:ample_cone_plane}, \cite[Lemma 2.4]{MS18:space_curves}}]
\label{lem:higherRankBound}
Let $A \into E$ be an injective morphism of $\nu_{\alpha, \beta}$-semistable objects in $\Coh^{\beta}(X)$ with $\nu_{\alpha, \beta}(A) = \nu_{\alpha, \beta}(E)$. If $W = W(A, E)$ is a semicircular wall and $\ch_0(A) > \ch_0(E) \geq 0$, then
\[
\rho_W^2 \leq \frac{\Delta(E)}{4 H^n \cdot \ch_0(A) (H^n \cdot \ch_0(A) - H^n \cdot \ch_0(E))}.
\]
\end{lem}

\subsection{Projective Space}

In the case of $\P^3$, further properties are known. To simplify notation, we use $\ch_i(E)$ in place of $H^{3 - i} \cdot \ch_i(E)$. Here, $H$ is the hyperplane class. The following inequality was conjectured in \cite{BMT14:stability_threefolds}. It was proven in \cite{Mac14:conjecture_p3} and brought to this precise form in \cite{BMS16:abelian_threefolds}.

\begin{thm}
\label{thm:p3_conjecture}
If $E \in \Coh^{\beta}(\P^3)$ is $\nu_{\alpha,\beta}$-semistable, then
\[
Q_{\alpha, \beta}(\P^3) := \alpha^2 \Delta(E) + 4\ch_2^{\beta}(E)^2 - 6\ch_1^{\beta}(E) \ch_3^{\beta}(E) \geq 0.
\]
\end{thm}

By no accident, the set of $(\alpha, \beta)$ with $Q_{\alpha, \beta}(E) = 0$ is a numerical wall, i.e., it is equivalent to
\[
\nu_{\alpha, \beta}(E) = \nu_{\alpha, \beta}(\ch_1(E), 2 \ch_2(E), 3\ch_3(E)).
\]
We call this wall $W_Q = W_Q(E)$. Accordingly, its radius is $\rho_Q = \rho_Q(E)$, and its center is $s_Q = s_Q(E)$. The crucial consequence of this theorem and Lemma \ref{lem:higherRankBound} is that if
\[
\rho_Q^2 > \frac{\Delta(E)}{4 H^n r(H^n r - H^n \cdot \ch_0(E))}
\]
for some positive integer $r$, then all walls are induced by a semistable subobject or quotient $A$ with the property $0 < \ch_0(A) < r$. We finish the section with some results on line bundles and derived duals.

\begin{prop}[{\cite[Proposition 4.1, 4.5]{Sch15:stability_threefolds}}]
\label{prop:line_bundles_uniquely_stable}
Assume $E \in \Coh^{\beta}(\P^3)$ is tilt-semistable and $n,m$ are integers with $m > 0$ such that either
\begin{enumerate}
\item $v = m \ch(\OO(n))$, or
\item $v = -m\ch_{\leq 2}(\OO(n))$.
\end{enumerate}
Then $E \cong \OO(n)^{\oplus m}$, respectively $E \cong \OO(n)^{\oplus m}[1]$.
\end{prop}

\begin{prop}[{\cite[Proposition 5.1.3]{BMT14:stability_threefolds}}]
\label{prop:tilt_derived_dual}
Let $E \in \Coh^{\beta}(\P^3)$ for $\beta \neq \mu(E)$ be a $\nu_{\alpha, \beta}$-semistable object. Then there exists a distinguished triangle 
\[
\tilde{E} \to \RlHom(E, \OO)[1] = \D(E) \to T[-1] \to \tilde{E}[1],
\]
where $\tilde{E} \in \Coh^{-\beta}(\P^3)$ is $\nu_{\alpha, -\beta}$-semistable, and $T$ is a zero-dimensional sheaf.
\end{prop}


\section{Some Chern character bounds}
\label{sec:some_bounds}

If $E$ is the ideal sheaf of a curve, then the following proposition is simply saying that its genus is bounded from above by the genus of a plane curve of the same degree.

\begin{prop}[{\cite[Proposition 3.2]{MS18:space_curves}}]
\label{prop:rank_one_old}
Let $E \in \Coh^{\beta}(\P^3)$ be $\nu_{\alpha, \beta}$-semistable. If either $H \cdot \ch(E) = (1,0,-d,e)$ or $H \cdot \ch(E) = (-1,0,d,e)$ holds, then
\[
e \leq \frac{d(d+1)}{2}.
\]
\end{prop}

For rank two sheaves we require the use of the Hirzebruch-Riemann Roch Theorem. Note that the Todd class of $\P^3$ is
\[
\td(T_{\P^3}) = \left(1,2,\frac{11}{6},1\right).
\]
If $S \subset \P^3$ is a smooth surface of degree $d$, then its Todd class is given by
\[
\td_S = \left(1, \left(2 - \frac{d}{2}\right)H, \frac{d^3}{6} - d^2 + \frac{11d}{6} \right).
\]

Bounds for the Chern characters of stable rank two sheaves are well known. This was first proved for stable reflexive sheaves in \cite[Theorem 1.1]{Har88:stable_reflexive_3}, and later generalized for all stable torsion free sheaves in \cite{OS85:spectrum_torsion_free_sheavesII}. We need the following version in tilt stability only for certain special cases.

\begin{prop}
\label{prop:rank_two_bounds}
For $\alpha > 0$, $\beta < \tfrac{c}{2}$, let $E \in \Coh^{\beta}(\P^3)$ be tilt-semistable with $\ch(E) = (2,c,d,e)$.
\begin{enumerate}
    \item If $c = 0$ and $d = 0$, then $e \leq 0$. If $e = 0$, then $E \cong \OO^{\oplus 2}$.
    \item If $c = 0$ and $d = -1$, then $e \leq 0$.
    \item If $c = 0$ and $d = -2$, then $e \leq 2$.
    \item If $c = -1$ and $d = -\tfrac{1}{2}$, then $e \leq \tfrac{5}{6}$. If $e = \tfrac{5}{6}$, then $E$ is destabilized in tilt-stability by an exact sequence
    \[
    0 \to \OO(-1)^{\oplus 3} \to E \to \OO(-2)[1] \to 0.
    \]
    \item If $c = -1$ and $d = -\tfrac{3}{2}$, then $e \leq \tfrac{17}{6}$.
\end{enumerate}
\end{prop}

\begin{proof}
A straightforward computation confirms that, except in the case $c = d = 0$, the bounds on $e$ are equivalent to $\chi(E) \leq 0$. In these cases, we can compute that under our restriction on $\beta$ the tilt-slope of $\OO$ is larger than the tilt-slope of $E$, and $\Hom(\OO, E) = 0$ follows. Note that $\Ext^2(\OO, E) = \Hom(E, \OO(-4)[1])$. We will show that $E$ is semistable below $W(E, \OO(-4)[1])$ to conclude.

If $c = -1$, then $\ch^{-1}_1(E) = 1$. This is the minimal positive value it can obtain, and since $\beta = -1$ is not the vertical wall, $E$ must be stable or unstable independently of $\alpha$ when $\beta = -1$. However, for $d = -\tfrac{1}{2}$ one can check that $\alpha = 0$, $\beta = -1$ describes a point inside or on $W(E, \OO(-4)[1])$. For $d = -\tfrac{3}{2}$ the wall $W(E, \OO(-4)[1])$ intersects the $\beta$-axis at $\beta = -4$ and $\beta = -1$. We are done unless $E$ is destabilized by a morphism $E \onto \OO(-4)[1]$. This means the semistable subobject has rank three. However,
\[
\frac{\Delta(E)}{4 \cdot 3} = \frac{7}{12} < \frac{9}{4} = \rho^2(E, \OO(-4)[1]),
\]
and we get a contradiction to Lemma \ref{lem:higherRankBound}. The exact sequence in case $\ch(E) = (2,-1,-\tfrac{1}{2}, \tfrac{5}{6})$ is a special case of \cite[Theorem 5.1]{Sch15:stability_threefolds}.
\begin{enumerate}
    \item Let $c = 0$ and $d = 0$. Note that $Q_{\alpha, \beta}(E) \geq 0$ is equivalent to $e \leq 0$ for arbitrary $\alpha > 0$. By Proposition \ref{prop:line_bundles_uniquely_stable} we get $E \cong \OO^{\oplus 2}$ in case $e = 0$.
    \item If $c = 0$ and $d = -1$, then $\nu_{0, \pm 1}(E) = 0$. If $E$ is destabilized along some semicircular wall, it happens along $\beta = -1$. Let $0 \to F \to E \to G \to 0$ induce such a semicircular wall containg a point $(\alpha, -1)$. Since $\ch^{-1}(E) = (2,2,0,e-\tfrac{2}{3})$, we must have $\ch_1^{-1}(F) = 1$. Write $\ch^{-1}_{\leq 2}(F) = (r, 1, x)$, where $r$ is an integer, and $x \in \tfrac{1}{2} + \Z$. If $r \leq 0$, we replace $F$ by $G$ in the following argument, and without loss of generality $r > 0$. We have
    \[
    -\frac{\alpha^2}{2} = \nu_{\alpha, -1}(E) = \nu_{\alpha, -1}(F) = x - \frac{\alpha^2}{2}r.
    \]
    This equation is equivalent to $(r - 1)\alpha^2 = 2x$. Since $x \in \tfrac{1}{2} + \Z$, we must have $r \geq 2$. Then together with $\Delta(F) \geq 0$ we get $0 < x \leq \tfrac{1}{2r} \leq \frac{1}{4}$, a contradiction.
    \item If $c = 0$ and $d = -2$, then $\ch^{-1}(E) = (2,2,-1, e - \tfrac{5}{3})$ holds. If there is a subobject $F \subset E$ that destablizes $E$ along $\beta = -1$, then $\ch^{-1}_1(F) = 1$.
    Let $\ch_{\leq 2}^{-1}(F) = (r,1,x)$. If we had $r = 1$, then a direct computation shows that $W(F, E)$ would be the vertical wall which is located at $\beta = 0$. Thus, $r \geq 2$ and the wall is given by the equation
    \[
    \alpha^2 = \frac{2x + 1}{r-1}.
    \]
    Hence, $x > -\tfrac{1}{2}$, and $\Delta(F) \geq 0$ implies $x \leq \tfrac{1}{2r} \leq \tfrac{1}{4}$. This is a contradiction to $x \in \tfrac{1}{2} + \Z$. We showed that no wall intersects $\beta = -1$. A straightforward computation shows that $\beta = -1$, $\alpha = 0$ describes a point inside $W(E, \OO(-4)[1])$, and $E$ has to be semistable at some point inside $W(E, \OO(-4)[1])$. \qedhere
\end{enumerate}
\end{proof}

An analogous statement for objects with rank $-2$ can be deduced.

\begin{cor}
\label{cor:rank_minustwo_bounds}
For $\alpha > 0$, $\beta > \tfrac{c}{2}$, let $E \in \Coh^{\beta}(\P^3)$ be $\nu_{\alpha, \beta}$-semistable with $\ch(E) = (-2,c,d,e)$.
\begin{enumerate}
    \item If $c = 0$ and $d = 0$, then $e \leq 0$.
    \item If $c = 0$ and $d = 1$, then $e \leq 0$.
    \item If $c = 0$ and $d = 2$, then $e \leq 2$.
    \item If $c = -1$ and $d = \tfrac{1}{2}$, then $e \leq \tfrac{5}{6}$.
    \item If $c = -1$ and $d = \tfrac{3}{2}$, then $e \leq \tfrac{17}{6}$.
\end{enumerate}
\end{cor}

\begin{proof}
By Proposition \ref{prop:tilt_derived_dual}, there is a distinguished triangle 
\[
\tilde{E} \to \D(E) \to T[-1] \to \tilde{E}[1],
\]
where $\tilde{E}$ is $\nu_{\alpha, -\beta}$-semistable, and $T$ is a zero-dimensional sheaf. Proposition \ref{prop:rank_two_bounds} applies to $\tilde{E}$ and the result follows immediately.
\end{proof}

\section{Walls for sheaves supported on surfaces}
\label{sec:walls}

Let $i: S \into \P^3$ be the embedding of a degree $c$ integral surface. We fix the hyperplane section $H$ as a polarization on $X$. Note that for smooth $S$ an object $E \in \Coh(S)$ is slope-(semi)stable if and only if $i_*E$ is $2$-GS-(semi)stable. Therefore, if $S$ is singular, we simply define slope stability in this way. The goal of this section is to prove the following proposition and its corollary.

\begin{prop}
\label{prop:rank_two_destabilizes_rank_two}
Let $E \in \Coh(S)$ be a slope-semistable sheaf with $\ch(i_*E) = (0, 2c, d, e)$. Assume that $i_* E$ is destabilized by a short exact sequence $0 \to F \to i_* E \to G \to 0$, where $F$ has rank one. 
Then the wall $W(i_* E, F)$ has to be smaller than or equal to the numerical wall $W(i_* E, \OO(\ch_1(F) - c)) = W(i_* E, \OO(-\ch_1(G) + c))$. 
\end{prop}

\begin{cor}
\label{cor:rank_two_destabilizes_rank_two}
Let $E \in \Coh(S)$ be a slope-semistable sheaf with $\ch(i_*E) = (0, 2c, d, e)$. If $i_{*}E$ is destabilized in tilt stability by a semicircular wall of radius $\rho$, then
\[
\rho \leq \frac{c}{2}.
\]
\end{cor}

\begin{proof}
Assume for a contradiction that $\rho(i_* E, F)$ is strictly larger than $\tfrac{c}{2}$. By Lemma \ref{lem:higherRankBound} this implies that the destabilizing subobject $F$ has rank one.
It is easy to compute that $s(i_* E, F) = \tfrac{d}{2c}$.
Then the wall contains points $(\alpha_0,\beta_0)$, $(\alpha_1, \beta_1)$, where $\beta_0 = \tfrac{d}{2c} - \tfrac{c}{2}$ and $\beta_1 = \tfrac{d}{2c} + \tfrac{c}{2}$. In particular,
\begin{align*}
0 < \ch_1^{\beta_1}(F) &= x - (\frac{d}{2c} + \frac{c}{2}), \\
2c = \ch_1^{\beta_0}(i_*E) &> \ch_1^{\beta_0}(F) = x - (\frac{d}{2c} - \frac{c}{2}).
\end{align*}
Therefore, $\tfrac{d}{2c} + \tfrac{c}{2} < x < \tfrac{d}{2c} + \tfrac{3c}{2}$. The radius of $W(i_* E, \OO(x - c))$ is given by
\[
\rho(i_* E, \OO(x - c))^2 = \left(x - c - \frac{d}{2c}\right)^2 \geq \rho(i_* E, F)^2.
\]
We get the contradiction $\rho(i_* E, F)^2 < \tfrac{c^2}{4}$.
\end{proof}

\begin{lem}
\label{lem:sub_or_quot_torsion_onS}
Under the hypothesis of Proposition \ref{prop:rank_two_destabilizes_rank_two} either the torsion part of the sheaf $F$ is supported on $S$ where it is a rank one torsion free sheaf, or $G$ is a two term complex, where $\HH^{-1}(G)$ is a line bundle, and $\HH^0(G)$ is supported on $S$ where it has rank one.
\end{lem}

\begin{proof}
By taking cohomology, we have the following exact sequence of sheaves
\[
0 \to \HH^{-1}(G) \to F \to E \to \HH^0(G) \to 0.
\]
By Lemma \ref{lem:h_minus_one_reflexive} the sheaf $\HH^{-1}(G)$ is reflexive, i.e., it is a line bundle. Assume both $\HH^0(G)$ is supported in dimension smaller than or equal to one, and $F$ has only torsion supported in dimension smaller than or equal to one. Let $T \subset F$ be the torsion part of $F$. Then the cokernel of $\HH^{-1}(G) \to F/T$ is the quotient of an ideal sheaf by a line bundle. Any such quotient is supported on a surface, where it has rank one. Since both $T$ and $\HH^0(G)$ are supported in dimension smaller than or equal to one, $E$ is also of rank one on a surface, contrary to assumption.


Assume $\HH^0(G)$ is supported in dimension larger than one. Since it is a quotient of $E$ and $S$ is integral, it has to be supported on $S$, where it has rank one. Indeed, if its rank was two, then $E$ and $\HH^0(G)$ would be isomorphic, a contradiction.

Assume that $F$ has torsion $T$ supported in dimension two. Since $\HH^{-1}(G)$ is a line bundle, its image has to be disjoint from $T$, and thus $T$ injects into $E$. Integrality of $S$ implies that $T$ is supported on $S$, where it is torsion-free of rank one or two. Assume for a contradiction that it has rank two. Then we get an exact sequence
\[
0 \to \HH^{-1}(G) \to F/T \to E/T.
\]
Since $F/T$ is torsion free and $\HH^{-1}(G)$ is a line bundle, the image in $E/T$ has rank at least one. Since $E$ is an extension of $T$ and $E/T$, it has rank at least three on $S$, a contradiction.
\end{proof}

\begin{lem}
\label{lem:subobject_has_torsion_on_S}
Assume $F$ is a $\nu_{\alpha, \beta}$-semistable object with $\ch_0 = 1$ that contains a subobject $T$ with $\ch_{\leq 1}(T) = (0,c)$. Then $(\alpha, \beta)$ has to be inside or on the wall $W(F, \OO(\ch_1(F) - c))$.
\end{lem}

\begin{proof}
Tensoring with $\OO(-\ch_1(F))$ reduces the statement to the case $\ch_1(F) = 0$. Let $d = -\ch_2(F)$ and $y = \ch_2(T)$. The fact $F \in \Coh^{\beta}(\P^3)$ implies $\beta  < 0$. For $\alpha \gg 0$ we have $\nu_{\alpha, \beta}(T) > \nu_{\alpha, \beta}(F)$ making $F$ unstable. Therefore, $(\alpha, \beta)$ must be located inside the wall $W(F, T)$. In particular, $s(F, T) < 0$. We can compute
\[
s(F, T) = \frac{y}{c}, \ s(F, \OO(-c)) = -\frac{c}{2} - \frac{d}{c}.
\]
Therefore, $s(F, T) \geq s(F, \OO(-c))$ is equivalent to $y \geq -\tfrac{c^2}{2} - d$ which is implied by $\Delta(F/T) \geq 0$.
\end{proof}

\begin{lem}
\label{lem:quotient_has_torsion_on_S}
Assume $G$ is a $\nu_{\alpha, \beta}$-semistable object for $\beta \neq \mu(G)$ with $\ch_0(G) = -1$ such that $\ch_{\leq 1}(\HH^0(G)) = (0,c)$. Then $(\alpha, \beta)$ has to be inside or on the wall $W(G, \OO(c - \ch_1(G)))$.
\end{lem}

\begin{proof}
Tensoring with $\OO(\ch_1(G))$ reduces the statement to the case $\ch_1(G) = 0$. Let $d = \ch_2(F)$ and $y = \ch_2(\HH^0(G))$. The fact $G \in \Coh^{\beta}(\P^3)$ implies $\beta > 0$. There is a surjective morphism $G \onto \HH^0(G)$ in $\Coh^{\beta}(\P^3)$. For $\alpha \gg 0$ we have $\nu_{\alpha, \beta}(G) > \nu_{\alpha, \beta}(\HH^0(G))$ making $G$ unstable. Therefore, $(\alpha, \beta)$ must be located inside the wall $W(G, \HH^0(G))$. In particular, $s(G, \HH^0(G)) > 0$. We can compute
\[
s(G, \HH^0(G)) = \frac{y}{c}, \ s(G, \OO(c)) = \frac{c}{2} + \frac{d}{c}.
\]
Hence, $s(G, \HH^0(G)) \leq s(G, \OO(c))$ is equivalent to $y \leq \tfrac{c^2}{2} + d$ which is implied by $\Delta(\HH^{-1}(G)) \geq 0$.
\end{proof}

\begin{proof}[Proof of Proposition \ref{prop:rank_two_destabilizes_rank_two}]
We write $\ch(F) = (1,0,-y,z) \cdot \ch(\OO(x))$. By Lemma \ref{lem:sub_or_quot_torsion_onS} we have to deal with two cases. We will start by showing 
\[
y \geq -\frac{c^2}{2} + cx - \frac{d}{2}.
\]
\begin{enumerate}
    \item Assume that the torsion part $T \subset F$ is a rank one torsion-free sheaf on $S$. By Lemma \ref{lem:subobject_has_torsion_on_S}, the wall $W(i_* E, F)$ is smaller than or equal to the wall $W(F, \OO(x - c))$. Therefore,
    \[
    \frac{d}{2c} = s(i_* E, F) \geq s(F, \OO(x - c)) = -\frac{c}{2} + x - \frac{y}{c}
    \]
    implies the lower bound on $y$.
    \item Assume that $\HH^0(G)$ is supported on $S$ where it has rank one. Note that $\ch_1(G) = \ch_1(i_* E) - \ch_1(F) = 2c - x$. Then Lemma \ref{lem:quotient_has_torsion_on_S} implies that $W(i_* E, G)$ is smaller than or equal to $W(G, \OO(x - c))$. Therefore, the lower bound on $y$ is implied by
    \[
    \frac{d}{2c} = s(i_* E, G) \leq s(G, \OO(x - c)) = \frac{c}{2} - x + \frac{d}{c} + \frac{y}{c}.
    \]
\end{enumerate}
This lower bound on $y$ shows
\[
\rho(i_* E, \OO(x - c))^2 = \frac{(2c^2 - 2cx + d)^2}{4c^2} \geq \frac{4c^2x^2 - 4cdx - 8c^2y + d^2}{4c^2} = \rho(i_* E, F)^2.
\]
We get that $W(i_* E, G) = W(i_* E, F)$ has to be smaller than or equal to $W(i_* E, \OO(x - c))$.
\end{proof}

\begin{cor}
\label{cor:bogomolov}
Let $X \subset \P^3$ be an integral hypersurface, and let $H$ be the hyperplane section. If $E$ is a rank two slope-semistable torsion free sheaf on $X$, then $\Delta_H(E) \geq 0$.
\end{cor}

\begin{proof}
Let $\ch_S(E) = (2,xH,zH^2)$. We have to show $x^2 - 4z \geq 0$.
A direct application of the Grothendieck-Riemann-Roch Theorem shows 
\[
\ch(i_* E) = \left(0, 2c, -c^2 + cx, \frac{c^3}{3} - \frac{c^2x}{2} + cz\right).
\]
By Corollary \ref{cor:rank_two_destabilizes_rank_two}, we know that $i_* E$ is semistable along its numerical wall with radius $\tfrac{c}{2}$. This wall contains the point $(\alpha_0, \beta_0) = (\tfrac{c}{2}, \tfrac{x}{2} - \tfrac{c}{2})$. The fact $Q_{\alpha_0, \beta_0}(i_*E) \geq 0$ is equivalent to the claim.
\end{proof}

\section{Main theorem}
\label{sec:main}

The goal of this section is to prove the following Theorem.

\begin{thm}
\label{thm:rank_two_bounds_surface}
Let $S \subset \P^3$ be a very general smooth projective surface of degree $d \geq 5$ over an algebraically closed field $\F$, and let $H$ be the hyperplane section on $S$. Further assume $E \in \Coh(S)$ is a slope-stable sheaf with $\ch_0(E) = 2$.
\begin{enumerate}
\item If $\ch_1(E) = -H$, then
\[
\Delta(E) \geq 3d^2 - 4d.
\]
Equality implies $h^0(E(H)) \geq 3$ and can be obtained for non-trivial extensions
\[
0 \to \OO_S(-H) \to E \to \II_Z \to 0,
\]
where $Z$ is a zero-dimensional subscheme of length $d - 1$ contained in a line.
\item If $\ch_1(E) = 0$, then
\[
\Delta(E) \geq 4d^2.
\]
Equality can be obtained for non-trivial extensions
\[
0 \to \OO_S(-H) \to E \to \II_Z(H) \to 0,
\]
where $Z$ is a zero-dimensional subscheme of length $2d$ contained in two non-intersecting lines such that $d$ points are contained in each line.
\end{enumerate}
\end{thm}

We start by proving the existence of objects with equality in the bound.

\begin{enumerate}
\item Using Serre duality on $S$ we get $\Ext^1(\II_Z, \OO_S(-H)) = H^1(\II_Z((d-3)H))$. There is a long exact sequence
\[
0 \to H^0(\II_Z((d-3)H)) \to H^0(\OO_S((d-3)H)) \to H^0(\OO_Z) \to H^1(\II_Z((d-3)H)) \to 0.
\]
The map $H^0(\OO_S((d-3)H)) \to H^0(\OO_Z)$ is given by evaluation at the points of $Z$. Since $Z$ is contained in a line, we can find a degree $d-3$ polynomial with arbitrary values at $d-2$ of its points, but the value at the last point is determined by these. Thus, the map is not surjective, and we get $H^1(\II_Z((d-3)H)) = \C$.

Assume that $E$ is not slope-stable. Then there is a slope-stable quotient $E \onto G$ of rank one such that $\mu_H(G) \leq \mu_H(E) = -\tfrac{1}{2}$. Thus, $\ch_1(G) = xH$ for some integer $x \leq -1$. Clearly, $\Hom(\II_Z, G) = 0$. If $x < -1$, then $\Hom(\OO_S(-H), G) = 0$, a contradiction. We must have $x = -1$. Since $\Hom(\OO_S(-H), G) \neq 0$, we get $G = \OO_S(-H)$, but then the morphism $E \onto \OO_S(-H)$ splits the exact sequence
\[
0 \to \OO_S(-H) \to E \to \II_Z \to 0.
\]
We assumed this extension was non-trivial, a contradiction.
\item Serre duality implies $\Ext^1(\II_Z(H), \OO_S(-H)) = H^1(\II_Z((d-2)H))$. As in the previous case, a polynomial of degree $d - 2$ on a line is determined by its value on $d - 1$ points. Since both lines contain $d$ points of $Z$, the map $H^0(\OO_S((d-2)H)) \to H^0(\OO_Z)$ is not surjective, and we get $H^1(\II_Z((d-2)H)) \neq 0$.

By \cite[Theorem 5.1.1]{HL10:moduli_sheaves} $E$ is a vector bundle. Indeed, the Cayley-Bacharach property holds precisely due to our choice of $Z$. Assume that $E$ is not slope-stable. Then there is a line bundle $F \into E$ such that $\mu(F) \geq \mu(E) = 0$. Thus, $\ch_1(F) = xH$ for an integer $x \geq 0$. We have $\Hom(F, \OO_S(-H)) = 0$. If $x \geq 1$, it is immediately clear that $\Hom(F, \II_Z(H)) = 0$, a contradiction. Assume $x = 0$. If $\Hom(F, \II_Z(H)) \neq 0$, then $Z$ is contained in a plane in $\P^3$ in contradiction to the assumptions on $Z$. \qedhere
\end{enumerate}


We will prove the bounds in the theorem by studying wall-crossing of $E$ as a torsion sheaf in tilt stability in $\P^3$. The following lemma will be necessary to handle some special walls.

\begin{lem}
\label{lem:dual_ideal_sheaf_line}
Let $E \in \Coh^{\beta}(\P^3)$ with $\ch_{\leq 2}(E) = (-1,0,1)$. Assume further that $E$ is $\nu_{\alpha, \beta}$-semistable for $\beta > 0$ and $\alpha \gg 0$. Then there is $n \leq 2$ together with a short exact sequence
\[
0 \to \OO[1] \to E \to \OO_L(n) \to 0.
\]
\end{lem}

\begin{proof}
Since $E$ is $\nu_{\alpha, \beta}$-semistable for $\beta > 0$ and $\alpha \gg 0$, Proposition \ref{prop:large_volume_limit} and Lemma \ref{lem:h_minus_one_reflexive} imply that $\HH^{-1}(E)$ is a reflexive sheaf with $\ch_{\leq 1} = (1,0)$, and $\HH^0(E)$ is supported in dimension one. Therefore, $\HH^{-1}(E) \cong \OO$, and $\HH^0(E)$ is supported on a line where it has rank one. This means $\HH^0(E)$ is the direct sum of a line bundle $\OO_L(m)$ for some $m \in \Z$ and a sheaf $T$ supported in dimension zero. 

Assume $T \neq 0$. There is a surjective morphism $E \onto \OO_L(m)$, and by the Snake Lemma the kernel is an extension between $T$ and $\OO[1]$. However, all such extensions are trivial, and there exists an injective morphism $T \into E$, in contradiction to stability.

We showed the existence of the sequence
\[
0 \to \OO[1] \to E \to \OO_L(n) \to 0.
\]
Since $E$ is tilt-semistable, we get $\Ext^1(\OO_L(n), \OO[1]) \neq 0$. However, using Serre duality, we get
\[
\Ext^1(\OO_L(n), \OO[1]) = H^1(\OO_L(n-4)) = H^0(\OO_L(2-n)),
\]
which is non-trivial if and only if $n \leq 2$.
\end{proof}

A simpler version of the same argument shows the following.

\begin{lem}
\label{lem:dual_ideal_sheaf_point}
Let $E \in \Coh^{\beta}(S)$ with $\ch_S(E) = (-1,0,1)$. Assume further that $E$ is $\nu_{\alpha, \beta}$-semistable for $\beta > 0$ and $\alpha \gg 0$. Then there is a point $P \in S$ with a short exact sequence 
\[
0 \to \OO_S[1] \to E \to \OO_P \to 0.
\]
\end{lem}

Note that the bounds in Theorem \ref{thm:rank_two_bounds_surface} can be equivalently stated as follows:

\begin{enumerate}
\item If $\ch_S(E) = (2,-H,e)$, then
\[
e \leq 1 - \frac{d}{2}.
\]
\item If $\ch_S(E) = (2,0,e)$, then
\[
e \leq -d.
\]
\end{enumerate}

\begin{lem}
\label{lem:ch_on_P3}
Let $E \in \Coh(S)$.
\begin{enumerate}
\item If $\ch_S(E) = (2,-H,e)$, then $\ch_{\P^3}(E) = (0, 2d, -d^2 - d, \tfrac{d^3}{3} + \tfrac{d^2}{2} + e)$.
\item If $\ch_S(E) = (2,0,e)$, then $\ch_{\P^3}(E) = (0, 2d, -d^2, \tfrac{d^3}{3} + e)$.
\end{enumerate}
\end{lem}

\begin{proof}
The statement follows from a direct application of the Grothendieck-Riemann-Roch Theorem.
\end{proof}

\begin{lem}
\label{lem:rank_sub_torsion_sheaf}
Let $E \in \Coh(S)$ be a slope-semistable sheaf with either
\begin{enumerate}
\item $\ch_S(E) = (2,-H,e)$ and $e \geq 1 - \frac{d}{2}$, or
\item $\ch_S(E) = (2,0,e)$ and $e \geq -d$.
\end{enumerate}
Then $E$ is destabilized in tilt stability in $\P^3$ along a semicircular wall by a subobject of rank one or two.
\end{lem}

\begin{proof}
\begin{enumerate}
\item By Lemma \ref{lem:ch_on_P3} we have $\ch_{\P^3}(E) = (0, 2d, -d^2 - d, \tfrac{d^3}{3} + \tfrac{d^2}{2} + e)$. This implies
\[
\rho^2_Q(E) = \frac{d^3 - 3d + 12e}{4d} \geq \frac{d^3 - 9d + 12}{4d} > \frac{d^2}{9} = \frac{\Delta(E)}{4 \cdot 3^2}.
\]
We can conclude by Lemma \ref{lem:higherRankBound}.
\item By Lemma \ref{lem:ch_on_P3} we have $\ch_{\P^3}(E) = (0, 2d, -d^2, \tfrac{d^3}{3} + e)$. This implies
\[
\rho^2_Q(E) = \frac{d^3 + 12e}{4d} \geq \frac{d^2}{4} - 3 > \frac{d^2}{9} = \frac{\Delta(E)}{4 \cdot 3^2}.
\]
We can conclude by Lemma \ref{lem:higherRankBound}. \qedhere
\end{enumerate}
\end{proof}

\begin{lem}
\label{lem:torsion_destabilized_by_rank_one}
Let $E \in \Coh(S)$ be a slope-stable sheaf that is destabilized by a short exact sequence $0 \to F \to E \to G \to 0$ in tilt stability in $\P^3$ with $\ch_0(F) = 1$.
\begin{enumerate}
\item If $\ch_S(E) = (2,-H,e)$, then $e < 1 - \frac{d}{2}$.
\item If $\ch_S(E) = (2,0,e)$, then $e < -d$. 
\end{enumerate}
\end{lem}

\begin{proof}
\begin{enumerate}
\item Let $\ch_S(E) = (2,-H,e)$, but $e \geq 1 - \frac{d}{2}$. A straightforward computation shows
\[
Q_{0,-1}(E) = Q_{0,-d}(E) = -2d^3 + 4d^2 - 12de \leq -2d^3 + 10d^2 - 12d < 0.
\]
Therefore, any wall contains two points $(\alpha_0, -1)$ and $(\alpha_1, -d)$. By construction of the category $\Coh^{\beta}(\P^3)$, we get $\ch_1(F) + 1 = \ch_1^{-1}(F) >  0$ and $-\ch_1(F) + d = \ch_1^{-d}(G) > 0$. Overall, we have
\[
\ch_1(F) \in [0, d - 1].
\]
By Proposition \ref{prop:rank_two_destabilizes_rank_two} $W(E, F)$ has to be smaller than or equal to $W(E, \OO(\ch_1(F) - d))$. We have $-d \leq \ch_1(F) - d \leq -1$. Therefore, along all such walls $Q_{\alpha, \beta}(E) < 0$, a contradiction.
\item Assume that $\ch_S(E) = (2,0,e)$, but $e \geq -d$. Let $\ch(F) = (1,0,-y,z)\cdot \ch(\OO(x))$. A straightforward computation shows
\[
Q_{0, -1}(E) = Q_{0, -d + 1}(E) = -4d^3 + 4d^2 - 12de \leq -4d^3 + 16d^2 < 0.
\]
Therefore, any wall contains two points $(\alpha_0, -1)$ and $(\alpha_1, -d + 1)$. By construction of $\Coh^{\beta}(\P^3)$, we get $x + 1 = \ch_1^{-1}(F) >  0$ and $-x + d + 1 = \ch_1^{-d + 1}(G) > 0$. Overall, we have
\[
x \in [0, d].
\]
By Proposition \ref{prop:rank_two_destabilizes_rank_two} the wall $W(E, F)$ has to be smaller than or equal to $W(E, \OO(\ch_1(F) - d))$. In order to be outside the area in which $Q_{\alpha, \beta}(E) < 0$, we must have either $\ch_1(F) - d > -1$ or $\ch_1(F) - d < -d + 1$. We are left with the possibilities $x = 0$ and $x = d$. 

Assume that $x = 0$. By Proposition \ref{prop:rank_one_old} we know $z \leq \tfrac{y(y+1)}{2}$. We get $y \in \{0, 1\}$ from
\[
\frac{d^2}{4} - 2y = \rho(E, F)^2 \geq \rho_Q(E)^2 = \frac{d^3 + 12e}{4d} > \frac{d^2}{4} - 3.
\]

If $y = 0$, then $F$ is the ideal sheaf of a zero-dimensional subscheme $Z \subset \P^3$. The non-derived restriction of $F$ to $S$ is the ideal sheaf of the scheme-theoretic intersection $Z \cap S$. Since $E$ is slope-stable, this restriction of $F$ cannot have a non-trivial morphism to $E$.

If $y = 1$, then $\ch(F) = (1,0,-1,z)$ and $F$ is an ideal sheaf of the union of a line with potentially embedded points and further points not on the line. By assumption $S$ is very general and does not contain any lines. Hence, the non-derived restriction of $F$ to $S$ is an ideal sheaf of points. Again this restriction cannot have a non-trivial map to $E$.


Assume that $x = d$. We get that either $y = d^2$ or $y = d^2 + 1$ from
\[
\frac{d^2}{4} \geq \rho(E, F)^2 = \frac{9d^2}{4} - 2y \geq \rho_Q(E)^2 = \frac{d^3 + 12e}{4d} \geq \frac{d^2}{4} - 3.
\]

We start with the case $y = d^2$. Then $\ch_{\leq 2}(G) = \ch_{\leq 2}(\OO(-d)[1])$, and therefore, Proposition \ref{prop:line_bundles_uniquely_stable} implies $G \cong \OO(-d)[1]$. By Lemma \ref{lem:sub_or_quot_torsion_onS} the sheaf $F$ contains a torsion sheaf $T$ supported on $S$. A straightforward computation shows $W(E, F) = W(F, \OO)$. In the proof of Lemma \ref{lem:subobject_has_torsion_on_S} we showed that the morphism $T \to F$ destabilizes $F$ above the wall $W(F, T)$ which is smaller than or equal to $W(F, \OO)$. This implies $W(F, T) = W(F, \OO)$ and the quotient $F/T$ has Chern character $\ch_{\leq 2}(F/T) = (1,0,0)$. This means it is an ideal sheaf of points $\II_Z$. By the Snake Lemma the quotient $E/T$ is the quotient of a map $\OO(-d) \to \II_Z$. Since $E$ is supported on $S$, we must have $E/T \cong \II_{(Z \cap S)/S}$. However, this is a contradiction to $E$ being slope-stable on $S$.

Assume that $y = d^2 + 1$. Then $\ch_{\leq 2}(G(d)) = (-1, 0, 1)$. Lemma \ref{lem:dual_ideal_sheaf_line} implies that there is a surjective map from $G$ onto a line bundle supported on a line. In particular, $E$ has such a map. Since $E$ is supported on $S$ this implies that $S$ contains a line, a contradiction. \qedhere
\end{enumerate}
\end{proof}

\begin{lem}
\label{lem:special_case_d5}
Assume $E \in \Coh(S)$ is slope-stable with $\ch_S(E) = (2,0,e)$, and $d = 5$. If $e > -5$, then $E$ is destabilized by a short exact sequence $0 \to F \to E \to G \to 0$ in tilt stability with $\ch_0(F) = 2$ and $\ch_1(F) = 0$.
\end{lem}

\begin{proof}
Assume that $\ch_S(E) = (2,0,e)$, but $e > -5$. The Chern character of $E$ in $\P^3$ is given by $\ch(E) = (0, 10, -25, \tfrac{125}{3} + e)$. By Lemma \ref{lem:rank_sub_torsion_sheaf} and Lemma \ref{lem:torsion_destabilized_by_rank_one} we know that $E$ is destabilized by a rank two subobject $F \into E$ in tilt stability in $\P^3$. Let $G$ be the quotient $E/F$, and let $\ch(F) = (2,x,y,z)$. A straightforward computation shows
\[
Q_{0,-1}(E) = Q_{0,-4}(E) = -60e - 400 < 0.
\]
Therefore, $0 < \ch^{-1}_1(F) = x + 2$ implies $x > -2$. Moreover, $x + 8 = \ch^{-4}_1(F) < \ch^{-4}_1(E) = 10$ implies $x < 2$. Overall, $x \in \{-1, 0, 1\}$.
\begin{enumerate}
\item Assume $x = -1$. Then $\frac{13}{4} \leq \rho^2_Q \leq \rho^2(E, F) = y + \tfrac{15}{4}$ implies $y \geq -\tfrac{1}{2}$. Since $\Delta(F) \geq 0$, we must have $y = -\tfrac{1}{2}$, and moreover, $Q_{\alpha, \beta}(E) \geq 0$ is equivalent to $e \leq -5$ for any $(\alpha, \beta)$ along the wall. 
\item Assume $x = 1$. Then $\frac{13}{4} \leq \rho^2_Q \leq \rho^2(E, F) = y + \tfrac{35}{4}$ implies $y \geq -\tfrac{11}{2}$. Together with $\Delta(G) \geq 0$, we get $y = -\tfrac{11}{2}$, and moreover, $Q_{\alpha, \beta}(E) \geq 0$ is equivalent to $e \leq -5$ for any $(\alpha, \beta)$ along the wall. \qedhere
\end{enumerate}
\end{proof}

\begin{proof}[Proof of Theorem \ref{thm:rank_two_bounds_surface}]
\begin{enumerate}
\item Let $\ch_S(E) = (2, -H, e)$, but $e \geq 1 - \tfrac{d}{2}$. By Lemma \ref{lem:rank_sub_torsion_sheaf} and Lemma \ref{lem:torsion_destabilized_by_rank_one} we know that $E$ is destabilized by a rank two subobject $F \into E$ in tilt stability in $\P^3$. Let $G = E/F$, and let $\ch(F) = (2,x,y,z)$. As in the proof of Lemma \ref{lem:torsion_destabilized_by_rank_one} we can compute $Q_{0,-1}(E) = Q_{0,-d}(E) < 0$. Therefore, any wall contains two points $(\alpha_0, -1)$ and $(\alpha_1, -d)$. By construction of $\Coh^{\beta}(\P^3)$ this implies $x + 2 = \ch_1^{-1}(F) > 0$ and $-x = \ch_1^{-d}(G) > 0$. Overall, we have $x = -1$. Next we can bound $y \in \{-\tfrac{3}{2}, -\tfrac{1}{2} \}$ by using
\[
\frac{d^2 - 9d + 12}{4d} \leq \frac{d^2 - 3d + 12e}{4d} = \rho^2_Q \leq \rho^2(E, F) = \frac{d^2}{4} + y - \frac{1}{4}. 
\]
Assume that $y = -\tfrac{1}{2}$. Then Proposition \ref{prop:rank_two_bounds} says $z \leq \tfrac{5}{6}$. A straightforward computation shows $\ch(G(d+1)) = (-2, -1, \tfrac{1}{2}, \tfrac{d}{2} + e - z + \tfrac{2}{3})$. Corollary \ref{cor:rank_minustwo_bounds} implies 
\[
e \leq \frac{1}{6} - \frac{d}{2} + z \leq 1 - \frac{d}{2}.
\]
Moreover, in case $e = 1 - \tfrac{d}{2}$, we get $z = \tfrac{5}{6}$, and Proposition \ref{prop:rank_two_bounds} implies that $F(H)$ has three linearly independent global sections. Thus, $E(H)$ also has these three global sections.

Assume that $y = -\tfrac{3}{2}$. Then Proposition \ref{prop:rank_two_bounds} says $z \leq \tfrac{17}{6}$. A straightforward computation shows $\ch(G(d+1)) = (-2, -1, \tfrac{3}{2}, \tfrac{3d}{2} + e - z + \tfrac{5}{3})$. Corollary \ref{cor:rank_minustwo_bounds} implies 
\[
e \leq \frac{7}{6} - \frac{3d}{2} + z \leq 4 - \frac{3d}{2} < 1 - \frac{d}{2}.
\]
\item Assume that $e > -d$. Since $\chi(E)$ is an integer, we have $e \geq -d + 1$. By Lemma \ref{lem:rank_sub_torsion_sheaf} and Lemma \ref{lem:torsion_destabilized_by_rank_one} we know that $E$ is destabilized by a rank two subobject $F \into E$ in tilt stability in $\P^3$. Let $G$ be the quotient $E/F$, and let $\ch(F) = (2,x,y,z)$. If $d = 5$, then Lemma \ref{lem:special_case_d5} says $x = 0$. If $d \geq 6$, a straightforward computation shows
\[
Q_{0,-\tfrac{1}{2}}(E) = Q_{0,-d+\tfrac{1}{2}}(E) = -2d^3 + d^2 - 12de \leq -2d^3 + 13d^2 - 12d < 0.
\]
Therefore, any wall contains two points $(\alpha_0, -1/2)$ and $(\alpha_1, -d + 1/2)$. By construction of $\Coh^{\beta}(\P^3)$ this implies $x + 1 = \ch_1^{-1/2}(F) > 0$ and $1 - x = \ch_1^{-d + 1/2}(G) > 0$. Overall, we get $x = 0$ regardless of $d$. We get $y \in \{0, -1, -2\}$ from
\[
\frac{d^3 - 12d + 12}{4d} \leq \frac{d^3 + 12e}{4d} = \rho^2_Q \leq \rho^2(E, F) = \frac{d^2}{4} + y.
\]

Assume that $y = 0$. Then $F$ is a slope-semistable sheaf with Chern character $(2,0,0,z)$. By \cite[Theorem 3.1]{BMSZ17:stability_fano} $F$ has to be strictly semistable. Any stable subobject $F$ with the same slope has $\Delta(F) = 0$. In particular, $F$ has a subobject that is an ideal sheaf of points or the structure sheaf. Such a map contradicts the slope-stability of $F$ on $S$.

If $y = -1$, then Proposition \ref{prop:rank_two_bounds} says $z \leq 0$. Since $\ch(G(d)) = (-2, 0, 1, d + e - z)$, we can use Corollary \ref{cor:rank_minustwo_bounds} to get $e \leq -d$. If $y = -2$, then Proposition \ref{prop:rank_two_bounds} says $z \leq 2$. Since $\ch(G(d)) = (-2, 0, 1, 2d + e - z)$, we can use Corollary \ref{cor:rank_minustwo_bounds} to get $e \leq -2d + 2 < -d$. \qedhere
\end{enumerate}
\end{proof}

\begin{cor}
\label{cor:moduli_on_surface}
Let $S \subset \P^3$ be a very general surface of degree $d \geq 5$ over an algebraically closed field of characteristic zero, and let $H$ be the hyperplane section on $S$. The moduli space of semistable rank two sheaves on $S$ with $\ch_1(E) = -H$ and $\Delta(E) = 3d^2 - 4d$ is given by $S$.
\end{cor}

\begin{proof}[Proof of Corollary \ref{cor:moduli_on_surface}]
We have $\ch(E) = (2,-H, 1  - \tfrac{d}{2})$. The first step is to analyze the tilt stability of $E$ on the surface $S$. By Theorem \ref{thm:rank_two_bounds_surface} $\hom(\OO_S(-H), E) \geq 3$. Therefore, $E$ is unstable below the numerical wall $W(E, \OO_S(-H))$. Since $\ch_1^{-1}(E) = H$, we know that $E$ is not destabilized along $\beta = -1$. This means, $E$ is destabilized along the wall $W = W(E, \OO_S(-H))$. We can compute
\[
\frac{\Delta_S(E)}{32d^2} = \frac{3}{32} - \frac{1}{8d} < \frac{1}{4} - \frac{1}{d} + \frac{1}{d^2} = \rho(E, \OO_S(-H))^2.
\]
By Lemma \ref{lem:higherRankBound} $E$ cannot be destabilized by a subobject of rank four or higher. Since $\OO_S(-H)$ is tilt-stable, any subobject of $\OO_S^{\oplus 3}(-H)$ in $\Coh^{\beta}(\P^3)$ with the same tilt-slope has to be $\OO_S^{\oplus k}(-H)$ for $1 \leq k \leq 3$. Three linearly independent morphisms $\OO_S(-H) \to E$, induce the short exact sequence
\begin{equation}
\label{eq:sequence_stable_objects}
0 \to \OO_S^{\oplus 3}(-H) \to E \to G \to 0,
\end{equation}
where $G$ is also tilt-semistable along the wall. Moreover, $\hom(\OO_S(-H), E) = 3$. Note that $\ch(G(2H)) = (-1,0,1)$. By Lemma \ref{lem:dual_ideal_sheaf_point} there is a point $P \in S$ and a short exact sequence
\[
0 \to \OO_S(-2H)[1] \to G \to \OO_P \to 0.
\]
Clearly, the moduli space of such objects $G$ is isomorphic to $S$ itself. From this sequence we can see that the derived dual of $G$ is $\II_P(2H)[-1]$. This means
\[
\Ext^1(G, \OO_S(-H)) = \Hom(G, \OO_S(-H)[1]) = H^0(\II_P(H)) = \C^3.
\]
In particular, $E$ being semistable means that $E$ is uniquely determined by $G$. Overall, we get a bijective morphism $M(2, -H, 1 - \tfrac{d}{2}) \to M(-1,0,1) \cong S$. Since our ground field has characteristic $0$, all we have to show is that $M(2, -H, 1 - \tfrac{d}{2})$ is smooth. We can compute
\[
\Ext^1(\OO_S(-H), G) = \Ext^1(\II_P(H)[-1], \OO_S) = H^0(\II_P((d-3)H))^{\vee}.
\]
Applying the functor $\Hom(\OO_S(-H), \cdot)$ to (\ref{eq:sequence_stable_objects}) shows $\Hom(\OO_S(-H), E) = \C^3$, and $\Ext^1(\OO_S(-H), E)$ is the kernel of the morphism
\[
H^0(\II_P((d-3)H))^{\vee} \to H^0(\OO_S((d-4)H)^{\oplus 3})^{\vee}.
\]
However, this morphism is injective since its dual is surjective, and we get $\Ext^1(\OO_S(-H), E) = 0$. We can compute $\Ext^1(G, \OO_S(-H)) = H^0(\II_P(H)) = \C^3$ and $\Ext^2(G, \OO_S(-H)) = 0$. Therefore, applying $\Hom(G, \cdot)$ to (\ref{eq:sequence_stable_objects}) shows $\Ext^1(G, E) = \C^{10}$. Finally, we apply $\Hom(\cdot, E)$ to (\ref{eq:sequence_stable_objects}) and obtain $\Ext^1(E, E) = \C^2$. Indeed, $M(2, -H, 1 - \tfrac{d}{2})$ is smooth.
\end{proof}

\def\cprime{$'$} \def\cprime{$'$}

\end{document}